\documentclass[12pt]{article}
\usepackage{graphicx}
\usepackage{amsfonts}
\usepackage{amsmath}
\usepackage{amssymb}
\usepackage{verbatim}
\usepackage[ansinew]{inputenc}
\usepackage{latexsym}
\usepackage{tabularx}
\usepackage{makeidx}
\usepackage{color}
\allowdisplaybreaks
\newtheorem{theorem}{Theorem}

\newtheorem{lemma}{Lemma}

\newtheorem{rem}{Remark}

\newcommand{\NN}{{\mathbb N}}

\newcommand{\PP}{{\mathbb P}}

\newcommand{\ZZ}{{\mathbb Z}}
\newcommand{\RR}{{\mathbb R}}
\newcommand{\CC}{{\mathbb C}}
\newcommand{\FF}{{\mathbb F}}

\newcommand{\bsb}{\boldsymbol{b}}
\newcommand{\bsk}{\boldsymbol{k}}
\newcommand{\bsy}{\boldsymbol{y}}
\newcommand{\bsz}{\boldsymbol{z}}
\newcommand{\bsx}{\boldsymbol{x}}

\newcommand{\bsq}{\boldsymbol{q}}
\newcommand{\bsw}{\boldsymbol{w}}
\renewcommand{\(}{\left (}
  \renewcommand{\)}{\right )}

\newcommand{\bsc}{\boldsymbol{c}}

\newcommand{\bsl}{\boldsymbol{l}}

\newcommand{\rmk}{{\rm k}}
\newcommand{\wal}{{\rm wal}}

\newcommand{\Log}{{\rm Log}}

\newcommand{\bszero}{{\boldsymbol{0}}}

\newenvironment{proof}{\begin{trivlist}
    \item[\hskip\labelsep{\it Proof.}]}{$\hfill\Box$\end{trivlist}}

\title{\scshape{On hybrid point sets stemming from Halton-type Hammersley point sets and polynomial lattice point sets}} 
\author{Roswitha Hofer\thanks{Institute of Financial Mathematics and Applied Number Theory, Johannes Kepler University Linz, Altenbergerstr. 69, 4040 Linz, AUSTRIA. roswitha.hofer@jku.at}}
\date{}

\begin{document}
\maketitle

\begin{abstract}
In this paper we consider finite hybrid point sets that are the digital analogs to finite hybrid point sets introduced by Kritzer. Kritzer considered hybrid point sets that are a combination of lattice point sets and Hammersley point sets constructed using the ring of integers and the field of rational numbers. In this paper we consider finite hybrid point sets whose components stem from Halton-type Hammersley point sets and lattice point sets which are constructed using the arithmetic of the ring of polynomials and the field of rational functions over a finite field. We present existence results for such finite hybrid point sets with low discrepancy. 
\end{abstract}

\section{Introduction and preliminaries}

This work is motivated by applications of the theory
of uniform distribution modulo one to numerical integration that is based on the Koksma--Hlawka inequality. This inequality states an upper bound for the integration error for a probably very high dimensional function $f:[0,1]^s\to\RR$ when using a simple, equally weighted quadrature rule with $N$ nodes $\bsz_0,\,\bsz_1,\ldots,\,\bsz_{N-1}$. More exactly,
$$\left|\int_{[0,1]^s}f(\bsz)d\bsz \,-\,\frac 1N \sum_{n=0}^{N-1}f(\bsz_n)\right|\leq V(f)D_N^*(\bsz_n).$$
Here $V(f)$ denotes the variation of $f$ in the sense of Hardy and Krause and $D_N^*(\bsz_n)$ denotes the \emph{star discrepancy} of the node set $\bsz_0,\,\bsz_1,\ldots,\,\bsz_{N-1}$ which is defined in the following. 

The star discrepancy $D^*_N$ of a point set $\mathcal{P}=(\bsz_n)_{n=0,1,\ldots, N-1}$ in $[0,1)^s$ is given by 
$$D^*_N(\mathcal{P})=D^*_N(\bsz_n)=\sup_J\left|\frac{A(J,N)}{N}-\lambda_s(J)\right|$$
where the supremum is extended over all half-open subintervals $J$ of $[0,1)^s$ with the lower left corner in the origin, $\lambda_s$ denotes the $s$-dimensional Lebesgue measure, and the counting function $A(J,N)$ stands for 
$$\#\{0\leq n<N:\bsz_n\in J\}.$$

We define $\Log(x):=\max(1,\log(x))$ for real numbers $x>0$. Furthermore we use the Landau symbol $h(N)=O(H(N))$ to express $|h(N)|\leq CH(N)$ for all $N\in\NN$ with some positive constant $C$ independent of $N$ and a function $H:\NN\to\RR^+$. If the implied constant $C$ depends on some parameters, then these parameters will appear as a subscript in the Landau symbol. A symbol $O$ without a subscript indicates, if nothing else is written, an absolute implied constant.

So far the best known upper bounds for the star discrepancy of concrete examples of point sets  $(\bsz_n)_{0\leq n<N}$ are of the form
$$ND^*_N(\bsz_n)=O(\Log^{s-1} N)$$
where the implied constant might depend on some parameters but is independent of $N$. Examples of such \emph{low-discrepancy point sets} are Hammersley point sets and $(t,m,s)$-nets. A slightly weaker discrepancy bound, i.e. $ND^*_N(\bsz_n)=O(\Log^{s} N)$, holds for good lattice point sets and good polynomial lattice point sets. 

Numerical integration based on low-discrepancy point sets, is well established as \emph{quasi-Monte Carlo} (qMC) method.
The
stochastic counterparts of quasi-Monte Carlo methods, namely \emph{Monte Carlo} (MC)
methods, work with sequences of pseudorandom numbers. For more details on qMC and MC integration and low-discrepancy point sets we refer to \cite{DP10} and \cite{niesiam}.

The potency of qMC methods and MC
methods for multidimensional numerical integration depends on the nature
and the dimensionality of the integrand. As a general rule of thumb, qMC methods are more effective in low dimensions and Monte Carlo methods work reasonably well in arbitrarily high dimensions.
This has led to the idea, first suggested and applied by Spanier \cite{spanier}, of melding the advantages
of qMC methods and MC methods by using so-called \emph{hybrid sequences}. 
The principle here is to sample a relatively small number
of dominating variables of the integrand by low-discrepancy sequences and
the remaining variables by pseudorandom sequences. Application of
hybrid sequences to challenging computational problems can be found in the literature (see e.g. \cite{DelChiLar, Oek2, OekTufBur, spanier}). 

In view of the Koksma--Hlawka inequality the analysis of numerical integration methods based on hybrid sequences requires the study of their discrepancy. 
There are probabilistic results on the discrepancy of hybrid sequences, e.g., in \cite{Gnewuch,Oek1}. Niederreiter \cite{NieAA09} was the first one who established nontrivial deterministic discrepancy
bounds for hybrid sequences, where the qMC components are Halton sequences or Kronecker sequences. Those results where improved, extended, and unified in a series of papers \cite{GomHofNie, Niederreiter10c, Niederreiter10, Niederreiter11b, Niederreiter11, NiederreiterWinterhof11}. In these and in several other papers, see e.g. \cite{DrmotaHoferLarcher,HellekalekKritzer,Hofer09a,HoferKritzer,hklp,hoflar,HoferLarcher,HoferPuchhammer,kritzer,KriPil12}, also hybrid sequences and hybrid point sets made by combining different qMC sequences were treated. The motivation is here to combine the advantages of different qMC point sets and sequences. The challenge is to handle the different structures of the qMC point sets and sequences when studying the discrepancy of such hybrid point sets and hybrid sequences.

In this paper we mention results of Kritzer \cite{kritzer} on hybrid point sets where the components stem from Hammersley point sets on the one hand, and lattice point sets in the sense of Hlawka \cite{hlawka} and Korobov \cite{korobov} on the other hand. 

For the definition of Hammersley point sets we need the radical inverse function $\varphi_{b}:\NN_0\to[0,1)$ where $b$ is a natural number greater or equal to $2$. To compute $\varphi_b(n)$ represent $n$ in base $b$ of the form $n=n_0+n_1b+n_2b^2+\cdots$ with $n_i\in\{0,1,\ldots,b-1\}$ and set $$\varphi_b(n)=\sum_{i=0}^\infty \frac{n_i}{b^{i+1}}.$$
For an $s$-dimensional Halton sequence $(\bsx_n)_{n\geq 0}$ (introduced in \cite{halton}) we choose $s$ pairwise coprime bases $b_1,\ldots,b_s\geq 2$ and set 
$$\bsx_n:=(\varphi_{b_1}(n),\ldots,\varphi_{b_s}(n)).$$
Now for an $(s+1)$-dimensional Hammersley point set we choose in addition a natural number $N$ and define the point set $(\bsy_n)_{0\leq n<N}$
by $$\bsy_n:=(n/N,\varphi_{b_1}(n),\ldots,\varphi_{b_s}(n)).$$

For a $t$-dimensional lattice point set $(\bsy_n)_{0\leq n<N}$ choose first a positive integer $N$ and $t$ integers $g_1,\ldots,g_t$. Then set 
$$\bsy_n:=\left(\left\{ng_1/N\right\},\ldots,\left\{ng_t/N\right\}\right),\, 0\leq n<N.$$
If $(g_1,\ldots,g_s)$ are of the specific form $(g,\ldots,g^{t})$ then we speak of a lattice point set of Korobov type. 

Kritzer ensured existence of lattice point sets and lattice point sets of Korobov type as well, such that they can be combined with Hammersley point sets, and the obtained hybrid point sets satisfy low discrepancy bounds. 

\begin{theorem}[{\cite[Theorem~1]{kritzer} }]
Let $s,t\in\NN$. Let $p_1,\ldots,p_s$ be distinct prime numbers and let $N$ be a prime number that is different from $p_1,\ldots,p_s$. Let $(\bsx_n)_{n\geq 0}$ be the Halton sequence in bases $p_1,\ldots,p_s$. Then there exist generating $g_1,\ldots,g_t\in\{1,\ldots, N-1\}$ such that the point set 
$$\mathcal{S}_N:=(n/N,\bsx_n,\bsy_n)_{0\leq n<N}$$ 
in $[0,1]^{1+s+t}$ with $\bsy_n:=(\{ng_1/N\},\ldots, \{ng_t/N\})$, satisfies 
$$ND^*_N(\mathcal{S}_N)=O(\log^{s+t+1} N)$$
with an implied constant independent of $N$. 
\end{theorem}

\begin{theorem}[{\cite[Theorem~3]{kritzer} }]
Let $s,t\in\NN$. Let $p_1,\ldots,p_s$ be distinct prime numbers and let $N$ be a prime number that is different from $p_1,\ldots,p_s$. Let $(\bsx_n)_{n\geq 0}$ be the Halton sequence in bases $p_1,\ldots,p_s$. Then there exists a generating $g\in\{1,\ldots, N-1\}$ such that the point set 
$$\mathcal{S}_N:=(n/N,\bsx_n,\bsy_n)_{0\leq n<N}$$ 
in $[0,1]^{1+s+t}$ with $\bsy_n:=(\{ng/N\},\ldots, \{ng^t/N\})$, satisfies 
$$ND^*_N(\mathcal{S}_N)=O(\log^{s+t+1} N)$$
with an implied constant independent of $N$. 
\end{theorem}

Kritzer used a slightly different lattice point set of Korobov type by setting $(g_1,\ldots,g_t)=(g,\ldots,g^t)$ instead of $(1,g,\ldots, g^{t-1})$. Note that $g_1=1$ won't mix well with the first component $n/N$. 

In the next section we will define the analogs to Hammersley point sets and lattice point sets that are using the arithmetics in the ring of polynomials and the field of rational functions over a finite field instead of the arithmetic in the ring of integers and the field of rational numbers, before we state two theorems that represent analogs to the two theorems of Kritzer.

\section{Halton-type Hammersley point sets, polynomial lattice point sets, and results on the star discrepancy of their hybrid point sets}

Let $p$ be a prime number. Let $\FF_p$ be the finite field with $p$ elements. Let $\FF_p[X]$ be the ring of polynomials over $\FF_p$, $\FF_p(X)$ the field of rational functions over $\FF_p$, and $\FF_p((X^{-1}))$ the field of formal Laurent series over $\FF_p$.

Let $s\in\NN$, and let $b_1(X)\ldots,b_s(X)$ be distinct monic pairwise coprime nonconstant polynomials over $\FF_p$ with degrees $e_1,\ldots,e_s$. We define the Halton type sequence $(\bsx_n)_{n\geq 0}$ in bases $(b_1(X),\ldots,b_s(X))$ by 
$$\bsx_n:=(\varphi_{b_1(X)}(n(X)),\ldots,\varphi_{b_s(X)}(n(X))).$$
Here $\varphi_{b(X)}(n(X))$ is the radical inverse function in the ring $\FF_p[X]$ defined as follows. Expand $n$ in base $p$, $n=n_0+n_1p+n_2p^2+\cdots$ with $n_i\in\{0,1,\ldots,p-1\}$ and associate the
 polynomial $n(X)=n_0X^0+n_1X+n_2X^2+\cdots$ where we do not distinguish between the set $\FF_p$ and the set $\{0,1,\ldots,p-1\}$. Now expand $n(X)$ in base $b(X)$ with $\deg(b(X))=e\geq 1$ as 
 $$n(X)=\rho_0(X)b^0(X)+\rho_1(X)b^1(X)+\rho_2(X)b^2(X)+\cdots$$
 with $\deg(\rho_j(X))<e$ for $j\in\NN_0$. Finally, define a bijection $$\sigma:\{\rho(X)\in\FF_p[X]:\deg(\rho(X))<e\}\to \{0,1,\ldots,p^e-1\}$$ and set $$\varphi_{b(X)}(n(X)):=\sum_{j=0}^\infty \frac{\sigma(\rho_j(X))}{p^{e(j+1)}}.$$
 To avoid technical effort we restrict to bijections $\sigma$ that are mapping $0$ to $0$.

Let $m \in \NN$. Using the Halton type sequence in bases $b_1(X)\ldots,b_s(X)$ we can define a $(s+1)$-dimensional Halton-type Hammersley point set of $N=p^m$ points by using the $n$th point of the form 
$$\left(\frac{n}{N},\bsx_n\right)$$
and letting $n$ range in $\{0,1,\ldots,N-1\}$. 

For the definition of polynomial lattice point sets we identify $\FF_p$ again with the set $\{0,1,\ldots,p-1\}$. 

Let $t,\,m\in\NN$. Let $p(X)\in\FF_p[X]$ be irreducible, monic, and with degree $m$. Furthermore, let $\bsq(X)=(q_1(X),\ldots,q_t(X))\in\FF^t_p[X]$. The $i$th component $y^{(i)}_n$ of the $n$th point $\bsy_n$ is computed as follows. 
Expand $\{{n(X)q_i(X)/p(X)}\}$ in its formal Laurent series 
$$\left\{\frac{n(X)q_i(X)}{p(X)}\right\}=\sum_{j=1}^\infty u_jX^{-j}$$
and evaluate it by exchanging $X$ with $p$ and summing up to the index $m$. Hence 
$$y_n^{(i)}=\sum_{j=1}^m u_jp^{-j}.$$

We can also compute the $i$th component $y_n^{(i)}$ of the $n$th point $\bsy_n$ by using the base $p$ representation of $n=\sum_{j=0}^\infty n_jp^j$ and a generating matrix $C_i\in\FF_p^{m\times m}$. Let $\sum_{j=1}^\infty a_jX^{-j}$ be the formal Laurent series of $\left\{\frac{q_i(X)}{p(X)}\right\}$. Define
$$C_i:=\begin{pmatrix} a_1&a_2&\cdots&a_m\\
a_2&a_3&\cdots&a_{m+1}\\
\vdots&\vdots&\cdots &\vdots\\
a_m&a_{m+1}&\cdots &a_{2m-1}\end{pmatrix}.$$
Compute $C_i\cdot(n_0,n_1,\ldots,n_{m-1})^T=(u_1,u_2,\ldots,u_m)^T\in\FF_q^{m}$
and set $$y_n^{(i)}=\sum_{j=1}^mu_jp^{-j}.$$

Finally letting $n$ range in the set $\{0,1,\ldots,p^m-1\}$ we obtain the polynomial lattice point set $\mathcal{P}(\bsq(X),p(X))=\{\bsy_0,\bsy_1,\ldots,\bsy_{p^m-1}\}\subset [0,1]^t$. 

If we choose $\bsq(X)$ of the specific form $(g(X),g^2(X),\ldots,g^{t}(X))$ with $g(X)\in\FF_p[X]$, we will speak of a polynomial lattice point set of Korobov type abbreviated to $\mathcal{K}(t,g(X),p(X))$. 

In the following two theorems we ensure existence of polynomial lattice point sets as well as polynomial lattice point sets of Korobov type, such that they can be combined with a Halton-type Hammersley point set and result in hybrid point sets satisfying low discrepancy bounds. Theorem~\ref{thm:1} represents an analog to \cite[Theorem~1]{kritzer} and Theorem~\ref{thm:2} is the pendant to \cite[Theorem~3]{kritzer}.

\begin{theorem}\label{thm:1}
Let $s,t\in\NN$ and $p\in\PP$, let $b_1(X),\ldots,b_s(X)$ be monic pairwise coprime nonconstant polynomials in $\FF_p[X]$ and $(\bsx_n)_{n\geq 0}$ be a Halton type sequence in bases $(b_1(X),\ldots,b_s(X))$. Furthermore, let $p(X)$ be a monic, irreducible polynomial in $\FF_p[X]$ of degree $m$, coprime with all base polynomials of the Halton-type sequence, and set $N=p^m$.  
Then there exists a $t$-tuple of polynomials $\bsq(X)\in\FF^t_p(X)$ with degrees $<m$ such that the star discrepancy $D^*_N$ of the point set $(n/p^m,\bsx_n,\bsy_n)_{0\leq n<p^m}\in[0,1]^{s+t+1}$ satisfies 
$$ND_N^*(n/p^m,\bsx_n,\bsy_n)=O_{b_1(X),\ldots,b_s(X),p,t}(\Log^{s+t+1} N).$$
Here $\{\bsy_0,\bsy_1,\ldots,\bsy_{p^m-1}\}$ is the polynomial lattice point set $\mathcal{P}(\bsq(X),p(X))$. 
\end{theorem}

\begin{theorem}\label{thm:2}
Let $s,t\in\NN$ and $p\in\PP$, let $b_1(X),\ldots,b_s(X)$ be monic pairwise coprime nonconstant polynomials in $\FF_p[X]$ and $(\bsx_n)_{n\geq 0}$ be a Halton type sequence in bases $(b_1(X),\ldots,b_s(X))$. Furthermore, let $p(X)$ be a monic, irreducible polynomial in $\FF_p[X]$ of degree $m$, coprime with all base polynomials of the Halton-type sequence, and set $N=p^m$. 
Then there exists a polynomial $g(X)$ over $\FF_p$ with degree $<m$ such that the star discrepancy $D^*_N$ of the point set $(n/p^m,\bsx_n,\bsy_n)_{0\leq n<p^m}\in[0,1]^{s+t+1}$ satisfies 
$$ND^*_N(n/p^m,\bsx_n,\bsy_n)=O_{b_1(X),\ldots,b_s(X),p,t}(\Log^{s+t+1} N).$$
Here $\{\bsy_0,\bsy_1,\ldots,\bsy_{p^m-1}\}$ is the polynomial lattice point set of Korobov type $\mathcal{K}(t,g(X),p(X))$. 
\end{theorem}

The rest of the paper is organized as follows. Section~\ref{sec:2} collects auxiliary results needed for the proofs of Theorem~\ref{thm:1} and Theorem~\ref{thm:2}, which are formulated in Section~\ref{sec:3} and Section~\ref{sec:4}.

\section{Auxiliary results}\label{sec:2}

From the construction of the Halton-type sequence we immediately obtain the following lemma. 

\begin{lemma}\label{lem:1}
Let $(\bsx_n)_{n\geq 0}$ be a Halton type sequence in pairwise coprime bases $b_1(X),\ldots,b_s(X)$ with degrees $e_1,\ldots,e_s$, and let 
$$I:=\prod_{i=1}^s\left[\frac{a_i}{p^{e_il_i}},\frac{a_i+1}{p^{e_il_i}}\right)$$
with $l_i\geq 0$, $0\leq a_i<p^{e_il_i}$ for $i=1,\ldots,s$. Then $\bsx_n\in I$ if and only if $$n(X)\equiv R(X)\pmod{ {\prod_{i=1}^{s}{b_i^{l_i}(X)}}}$$
where the $R(X)$ depends on the $a_i$ and $\deg(R(X))<\sum_{i=1}^se_il_i$. Furthermore there is a one-to-one correspondence between all possible choices for $a_1,\ldots,a_s$ and $R(X)$. 
\end{lemma}

Let $e(x):=\exp(2\pi \sqrt{-1} x)$ for $x\in\RR$. 
We define the $\bsk$th Walsh function $\wal_{\bsk}$ in base $p$ on $[0,1)^t$ as follows. Let $\Phi_0:\{0,1,\ldots,p-1\}\to \{z\in\CC:|z|=1\},\,a\mapsto e(a/p)$. Note that for a given $b\in\{0,1\ldots, p-1\}$ we have that $\sum_{a=0}^{p-1}(\Phi_0(a))^b$ equals $p$ if $b=0$ and $0$ else. 

The $k$th Walsh function $\wal_k,$ for $k\geq 0$, to the base $p$ is defined by 
$$\wal_k(x):=\prod_{j=0}^\infty(\Phi_0(x_j))^{k_j}$$
where $x=x_0x_1\ldots$ is the base $p$ expansion of $x\in[0,1)$ and $k=\sum_{j=0}^\infty k_jp^j$ is the base $p$ expansion of $k\in\NN_0$. 
For vectors $\bsk=(k_1,\ldots,k_t)\in\NN_0^t$ and $\bsx=(x_1,\ldots,x_t)\in[0,1)^t$ the Walsh function $\wal_{\bsk}$ on $[0,1)^t$ denotes $$\wal_{\bsk}(\bsx):=\prod_{i=1}^t\wal_{k_i}(x_i).$$

\begin{lemma}[{\cite[Theorem~1]{hell}}]\label{lem:2}
Let $\mathcal{P}=\{\bsy_0,\bsy_1,\ldots,\bsy_{N-1}\}$ be a finite point set in $[0,1)^t$ with $\bsy_n$ of the form $\bsy_n=\{\bsw_n/M\},\, \bsw_n\in\ZZ^t$. Suppose that $M=p^m$, where $m$ is positive integer. Then the following estimate holds:
$$D^*_N(\bsy_n)\leq \underbrace{1-(1-1/M)^t}_{\leq t/M}+\sum_{\bsk\in\Delta_m^*}\rho_{\wal}(\bsk)|S_N(\wal_{\bsk})|,$$
where 
$$ S_N(\wal_{\bsk}):=\frac{1}{N}\sum_{n=0}^{N-1}\wal_{\bsk}(\bsy_n),$$
$$\Delta_m:=\{\bsk\in\ZZ^t:0\leq k_i<p^m,\,\mbox{ for }i=1,\ldots,t\}$$
$\Delta^*_m=\Delta_m\setminus\{\bszero\}$, and 
$$\rho_{\wal}(\bsk):=\prod_{i=1}^t\rho_{\wal}(k_i)$$
with 
$$\rho_{\wal}(k)=\left\{\begin{array}{ll}
1&\mbox{ if } k=0,\\
\frac{1}{p^{g+1}\sin \pi k_g/p}&\mbox{ if }p^g\leq k<p^{g+1},\,g\geq 0.
\end{array}\right.,$$
where $k_g$ is the $g$th digit of $k$ in the base $p$ expansion of $k$. 
\end{lemma}

\begin{lemma}[{\cite[Lemma 10.22]{DP10}}]\label{lem:3}
Let $t,\,m\in\NN$. For any prime number $p$, we have 
$$\sum_{\bsk\in \Delta_m}\rho_{\wal}(\bsk)=\left(1+m\frac{p^2-1}{3p}\right)^t.$$
\end{lemma}

For the statement of the next auxiliary result we define the following magnitudes:

$$G_{p,m}=\{f(X)\in\FF_p[X]:\deg(f(X))<m\}\quad\mbox{and}\quad G^*_{p,m}=G_{p,m}\setminus\{0\}.$$
Furthermore for the rational function $p(X)/q(X)$ in $\FF_p(X)\setminus\{0\}$ we define the degree evaluation $\nu$ by $$\nu(p(X)/q(X)):=\deg(p(X))-\deg(q(X))$$
and we set $\nu(0)=-\infty$. 

%We set 
%$$\mathcal{D}_{\bsq,p}=\{\bsk\in\{0,1,\ldots,p^{m}-1\}^t: \bsk(X)\bsq(X)\equiv 0 \pmod{p(X)}\} \quad \mbox{and}\quad \mathcal{D}_{\bsq,p}':=\mathcal{D}_{\bsq,p}\setminus\{\bszero\}.$$

\begin{lemma} \label{lem:4}
Let $p(X)$ be a monic irreducible polynomial in $\FF_p(X)$ with degree $m$. Let $u\in\NN_0$ such that $u\leq m$. Then 
$$\#\{a(X)\in G_{p,m}^*:\nu(a(X)/p(X))<-u\}\leq p^{m-u}-1. $$
\end{lemma}
\begin{proof}
The restriction $\nu(a(X)/p(X))<-u$ means $\deg(a(X))< \deg(p(X))-u=m-u$ and the result follows. 
\end{proof}

%\begin{lemma}[{\cite[Lemma~2.5]{niesiam}}]\label{lem:5}
%Let $\bsx_{1},\ldots,\bsx_N,\bsy_1,\ldots,\bsy_n\in[0,1]^t$ satisfy $\max_{1\leq i\leq t}|y^{(i)}_{n}-x_n^{(i)}|\leq1/q^m$ for $1\leq n\leq N$. 
%Then 
%$$|D_N(\bsx_{1},\ldots,\bsx_N)-D_N(\bsy_1,\ldots,\bsy_n)|\leq 2^t/q^m$$
%\end{lemma}
%\begin{proof}
%The proof is along the lines of the proof in \cite[Proof of Lemma~2.5]{niesiam}.
%\end{proof}

With a number $k\in\{0,1,\ldots,p^m-1\}$ we associate the polynomial $k(X)=\sum_{i=0}^{m-1}{k_i}X^i$ where the coefficients are determined by the base $p$ representation of $k=\sum_{i=0}^{m-1}k_ip^i$. 
For a tuple $\bsk\in \Delta_m$ we associate a polynomial with each component and write $\bsk(X)$ for the $t$-tuple of polynomials.

\begin{lemma}[{\cite[Lemma~1]{hofer}}]\label{lem:6}
Let $e\in\NN_0$, $B(X),R(X)\in\FF_p[X]$ with $\deg(R(X))<\deg(B(X))=e$, and let $B(X)$ be monic. Furthermore, let $u\in\NN$ and $K\in\NN_0$. Let $n=Kp^{u+e}, Kp^{u+e}+1,\ldots,(K+1)p^{u+e}-1$. We regard all associated polynomials $n(X)$ that satisfy $n(X)\equiv R(X)\pmod{B(X)}$. Then they are of the form 
$$n(X)=k(X)B(X)+R(X)$$
with $k(X)$ out of the set 
$$k(X)=r(X)+X^uC(X)$$
with a fixed $C(X)\in\FF_p[X]$ and $r(X)$ ranges over all polynomials of degree $<u$. \end{lemma}

%\begin{lemma}[{\cite[Theorem~5.34]{DP10}}] \label{lem:7}
%Let $s,m\in\NN$ and $p$ be a prime. For the star discrepancy of the digital net $\mathcal{P}$ over $\FF_p$ with generating matrices $C_1,\ldots,C_s\in\FF_p^{m\times m}$, we have 
%$$D^{*}_{p^m}\leq 1-(1-\frac{1}{p^m})^s+\sum_{\bsk\in\mathcal{D}'_m}\rho_{\wal}(\bsk)$$
%where 
%$$\bsk\in\mathcal{D}'_m=\{\bsk\in\{0,1,\ldots,p^m-1\}^s:C_1^T\rmk_1+\cdots+C_s^T\rmk_s=\bszero\}\setminus\{\bszero\}$$
%here $\rmk_i$ denotes the $m$-dimensional column vector built up by the base $p$ digits of the $i$th component of $\bsk$. 
%\end{lemma}

%\begin{lemma}[{\cite[Proposition~10.4]{DP10}}]\label{lem:8}
%Let $p$ be a prime. For $p(X)\in\FF_p[X]$ monic of degree $m$ of the specific form $p(X)=X^m+a_1X^{m-1}+\cdots a_{m-1}X+a_m$ and $q(X)=q_1X^{m-1}+\cdots q_{m-1}X+q_m$, the coefficients $u_l$, $l\in\NN$ in the Laurentseries expansion of 
%$$\frac{q(X)}{p(X)}=\sum_{l=1}^\infty u_lX^{-l}$$
%can be computed as follows: the first $m$ coefficients $u_1,\ldots,u_m$ are obtained by solving the linear system 
%$$
%\begin{pmatrix}
%1&0&\cdots&\cdots &0\\
%a_1&1&\ddots && \vdots \\
%\vdots&a_1&\ddots&\ddots &\vdots\\
%a_{m-2}&&\ddots&\ddots&0\\
%a_{m-1}&a_{m-2}&\cdots&a_1&1 
%\end{pmatrix}
%\begin{pmatrix}u_1\\u_2\\\vdots\\u_m
%\end{pmatrix}=
%\begin{pmatrix}q_1\\q_2\\\vdots\\q_m
%\end{pmatrix}
%$$
%and for $l>m$, $u_l$ is obtained from the linear recursion in $\FF_p$,
%$$u_l+u_{l-1}a_1+u_{l-2}a_2+\cdots+u_{l-m}a_m=0.$$
%\end{lemma}

\begin{lemma}[{\cite[Theorem~2.6]{KuiNie}}]\label{lem:9}
For $1\leq i \leq k$ let $w_i$ be a point set of $N_i$ elements in $[0,1]^s$. Let $w$ be the superposition of $w_1,\ldots,w_k$, that is a point set of $N=N_1+\cdots+N_k$ points. Then 
$$ND_N^*(w)\leq \sum_{i=1}^kN_iD^*_{N_i}(w_i).$$
\end{lemma}

\section{Proof of Theorem~\ref{thm:1}}\label{sec:3}

In this section we investigate the distribution of the point set $(\bsz_n)_{0\leq n<p^m}\in[0,1)^{1+s+t}$ with $m,\,s,\,t\in\NN$ and
$$\bsz_n:=(n/p^m,\bsx_n,\bsy_n)$$
where $(\bsy_n)_{0\leq n<p^m}$ is a polynomial lattice point set $\mathcal{P}(\bsq(X),p(X))$ in $[0,1)^t$ with $p(X)$ monic, irreducible, and with degree $m$ and where $(\bsx_n)_{n\geq 0}$ is a Halton-type sequence in bases $(b_1(X),\ldots,b_s(X))$, all monic, pairwise coprime, coprime with $p(X)$, and with degrees $e_1,\ldots,e_s$.

We set $N=p^m$. Using a well-known result in discrepancy theory (see, e.g. \cite[Lemma~3.7]{niesiam}), we have 
$$ND^*_N(\bsz_n)\leq \max_{1\leq \tilde{N}\leq N}\tilde{N}D^*_{\tilde{N}}((\bsx_n,\bsy_n))+1.$$

Let $\tilde{N}\in\{1,\ldots,N\}$ be fixed. 

We expand $\tilde{N}$ in base $p$, $\tilde{N}=N_0+N_1p+\cdots+N_rp^r$ with $N_i\in\{0,1,\ldots,p-1\}$ and $r\leq m$. For $u=0,\ldots,r$ and $v=1,\ldots,N_u$ we define the point set 
$$w_{u,v}:=\{(\bsx_n,\bsy_n):n\in\NN_0,\,\,(v-1)p^u+\cdots+N_rp^r\leq n< {v}p^u+\cdots+N_rp^r\}.$$
Then $|w_{u,v}|=p^u$ and  $$\{0,1,\ldots,\tilde{N}-1\}$$ is obtained by the disjoint union $$\bigcup_{u=0}^r\bigcup_{v=1}^{N_u}\{n\in\NN_0:(v-1)p^u+\cdots+N_rp^r\leq n< {v}p^u+\cdots+N_rp^r\}$$ 
of at most $pm=p\log_p N$ sets. 

We apply Lemma~\ref{lem:9}, which results in one $\log \,N$ factor in Theorem~\ref{thm:1} with a  constant depending on $p$. Then we have have to estimate 
$$p^uD_{p^u}^*(w_{u,v}).$$

We define 
$$f_i:=\left\lceil \frac{u}{e_i}\right\rceil$$  
for $1\leq i\leq s$. 
  
  The first aim in the proof is to compute or
  estimate the counting function $A(J,p^u)$ relative to the pointset $w_{u,v}$,
where $J\subseteq [0,1)^{s+t}$ is an interval of the form
\begin{equation} \label{Equ:formJ}
  J=\prod_{i=1}^s[0,v_ip^{-e_if_i})\times
  \prod_{j=1}^t[0,\beta_j)
\end{equation}
 with $v_1,\ldots,v_s \in\ZZ,\,
  1\leq v_i\leq p^{e_if_i}$ for $1\leq i\leq s$, and $0
  < \beta_j \leq 1$ for $1\leq j\leq t$.

  The crucial step is to exploit special properties of the Halton-type sequence. By Lemma~\ref{lem:1},
 for any integer $n\ge 0$ we have
  \begin{equation*} (\varphi_{b_1(X)}(n(X)),\ldots,\varphi_{b_s(X)}(n(X)))\in\prod_{i=1}^{s}[0, v_i p^{-e_if_i})
    \text{ if and only if } n(X)\in \bigcup_{k=1}^{M}\mathcal{R}_k,
  \end{equation*}
  where $$1\le M\le p^{e_1} \cdots p^{e_s} f_1 \cdots f_s=O_{p,b_1(X),\ldots,b_s(X)}(\log^sN).$$ Each $\mathcal{R}_k$ is a residue
  class in $\FF_p[X]$, and $\mathcal{R}_1,\ldots, \mathcal{R}_M$ are (pairwise) disjoint.  The moduli $B_k(X)$ of
  the residue classes $\mathcal{R}_k$ are of the form $b_1(X)^{j_1} \cdots b_s(X)^{j_s}$
  with integers $1\le j_i\le f_i$ for $1\le i\le s$ and the residues $R_k(X)$ satisfy $\deg(R_k(X)) <\deg(B_k(X))$ for $1 \le k \le M$.  The sets $\mathcal{R}_1,\ldots,
  \mathcal{R}_M$ depend only on $b_1(X),\ldots, b_s(X),v_1,\ldots, v_s,f_1,\ldots,
  f_s$ and are thus independent of $n$. Furthermore, one can easily prove for the Lebesgue measure of
  $\prod_{i=1}^s[0,v_ip^{-e_if_i})$ that
  \begin{eqnarray*}
  \lefteqn{\lambda_s\(\prod_{i=1}^s[0,v_ip^{-e_if_i})\)}\\
  &=&\prod_{i=1}^{s}v_ip^{-e_if_i} \\
    &=&\lim_{N\to\infty}\#\{0\leq n<N:(\varphi_{b_1(X)}(n(X)),\ldots,\varphi_{b_s(X)}(n(X)))\in \prod_{i=1}^s[0,v_ip^{-e_if_i}) \}\\
    &=&\lim_{N\to\infty}\#\{0\leq n<N: n(X)\in \bigcup_{k=1}^{M}\mathcal{R}_k\} \\
    &=&\sum_{k=1}^M\lim_{N\to\infty}\#\{0\leq n<N:n(X)\equiv R_k(X) \pmod{B_k(X)}\}\\
    &=&    \sum_{k=1}^{M}\frac{1}{p^{\deg(B_k(X))}},
  \end{eqnarray*}
  by applying the uniform distribution of the Halton type sequence and the disjointness of $\mathcal{R}_1,\ldots,\mathcal{R}_M$.

  Now we split up the counting function $A(J,p^u)$ into $M$ parts as follows:
  $A(J,p^u)=\sum_{k=1}^{M} S_k$, where
  \begin{eqnarray*}
    S_k&=&\#\{(v-1)p^u+\cdots+N_rp^r\leq n< {v}p^u+\cdots+N_rp^r:n(X)\equiv R_k(X) \pmod{B_k(X)} \\
   & &\quad \quad \text{ and }
      \bsy_n\in \prod_{j=1}^t[0,\beta_j)\}
  \end{eqnarray*}
 for $1\leq k\leq M$. 
Then $$|A(J,p^u)-p^u\lambda_{s+t}(J)|\leq \sum_{k=1}^{M}\underbrace{\left|S_k-p^u\frac{1}{p^{\deg(B_k(X))}}\prod_{j=1}^t \beta_j \right|}_{=:\delta_k}.$$ 
This summation over $k$ then results in $s$ $\log\, N$ factors in Theorem~\ref{thm:1} with a constant depending in $p,b_1(X),\ldots,b_s(X)$. 
  
  We fix $k$ with $1 \le k \le M$ for the moment.

Note that if $p^u < p^{\deg(B_k(X))}$, then $S_k=0$ or $1$, and so in this case
$\delta_k\leq 1$.

Assume now that $p^u \ge p^{\deg(B_k(X))}$. 
We define the set 
$$\mathcal{L}_k:=\left\{(v-1)p^u+\cdots+N_rp^r\leq n< {v}p^u+\cdots+N_rp^r:n(X)\equiv R_k(X) \pmod{B_k(X)}\right\}.$$
Then by Lemma~\ref{lem:6}, we know $|\mathcal{L}_k|=p^{u-\deg(B_k(X))}=:L_k$. 
We define the point set 
$$\mathcal{P}_k=\{\bsy_n:n\in\mathcal{L}_k\}.$$
Then 
  \begin{align*}
    \delta_k\leq L_kD^*_{L_k}(\mathcal{P}_k).
  \end{align*}
  
We summarize 
  \begin{eqnarray*}
   | A(J,p^u)-p^u\lambda_{s+t}(J)| \leq O(M)+\sum_{k=1\atop \deg(B_k(X))\leq u}^M   L_kD^*_{L_k}(\mathcal{P}_k)
  \end{eqnarray*}
  An arbitrary interval $I\subseteq [0,1)^{s+t}$ of the form
\begin{equation} \label{Equ:formI}
I= \prod_{i=1}^s[0,\alpha_i)\times \prod_{j=1}^t[0,\beta_j)
\end{equation}
 with
  $0<\alpha_i\leq 1$ for $1\leq i\leq s$ and $0 <
  \beta_j \leq 1$ for $1\leq j\leq t$ can be approximated from below
  and above by an interval $J$ of the form~\eqref{Equ:formJ},
  by taking the nearest fraction to the left and to the right, respectively, of $\alpha_i$ of the form
$v_iq^{-e_if_i}$ with $v_i\in\ZZ$. We easily get
  \begin{equation*}
    \left|A(I,p^u) -p^u\lambda_{s+t}(I)\right|\leq
    \underbrace{p^u\sum_{i=1}^s p^{-e_if_i}}_{\leq s}+\left|A(J,N) -N\lambda_{s+t}(J)\right|.
  \end{equation*}

The core of the proof is the study of the average
\begin{equation}\label{equ:central}
\frac{1}{|(G^*_{p,m})|^t}\sum_{\bsq(X)\in (G^*_{p,m})^t } L_kD^*_{L_k}(\mathcal{P}_k),\end{equation}
after exchanging the order of summation.

Note that $B_k(X)$ is monic and coprime with $p(X)$, and $\deg(B_k(X))\leq u$. In the following we set $d:=u-\deg(B_k(X))$ and we will omit the index $k$.

First we compute the subset $\mathcal{P}$ of the polynomial lattice point set $\mathcal{P}(\bsq(X),p(X))$ using the corresponding $n(X)\equiv R(X)\pmod{B(X)}$ and bearing in mind Lemma~\ref{lem:6}.

 Choose $l\in\{0,1,\ldots,p^{d}-1\}$, regard
\begin{eqnarray*}
\lefteqn{\left\{\frac{\Big(\big(l(X)+X^dC(X)\big)B(X)+R(X)\Big)q_i(X)}{p(X)}\right\}}\\
&=&\left\{\frac{l(X)B(X)q_i(X)}{p(X)}\right\}+\left\{\frac{(X^dC(X)B(X)+R(X))q_i(X)}{p(X)}\right\}.
\end{eqnarray*}
Let $\sum_{j=1}^\infty r^{(i)}_jX^{-j}$ be the formal Laurent series of $\left\{\frac{(X^dC(X)B(X)+R(X))q_i(X)}{p(X)}\right\}$ and $\sum_{j=1}^\infty a^{(i)}_jX^{-j}$ be the Laurent series of $\left\{\frac{B(X)q_i(X)}{p(X)}\right\}$
then compute 
$$\begin{pmatrix}
a^{(i)}_1&\cdots&a^{(i)}_d\\
\vdots&\ddots&\vdots\\
a^{(i)}_d&\cdots &a^{(i)}_{2d-1}\\
\vdots&\ddots&\vdots\\
a^{(i)}_m&\cdots&a^{(i)}_{m+d-1}
\end{pmatrix} 
\underbrace{\begin{pmatrix}l_0\\l_1\\\vdots\\l_{d-1}
\end{pmatrix}}_{=:\bsl}+
\begin{pmatrix}r^{(i)}_1\\\vdots\\r^{(i)}_d\\\vdots\\r^{(i)}_{m}
\end{pmatrix}=\begin{pmatrix}y^{(i)}_{l,1}\\\vdots\\y^{(i)}_{l,d}\\\vdots\\y^{(i)}_{l,m}
\end{pmatrix}\in\FF_p^m.
$$
Set
$$y^{(i)}_{n_l}=\sum_{j=1}^my_{l,j}^{(i)}p^{-j}.$$
Finally, letting $l$ range between $0$ and $p^d-1$.

We define 
$$C_{i,d}=\begin{pmatrix}
a^{(i)}_1&\cdots&a^{(i)}_d\\
\vdots&\ddots&\vdots\\
a^{(i)}_d&\cdots &a^{(i)}_{2d-1}\\
\vdots&\ddots&\vdots\\
a^{(i)}_m&\cdots &a^{(i)}_{m+d-1}
\end{pmatrix}. $$

We apply Lemma~\ref{lem:2} to $LD^*_L(\mathcal{P})$ and obtain 
\begin{align*}
\frac{1}{|(G^*_{p,m})|^t}\sum_{\bsq(X)\in (G^*_{p,m})^t } LD^*_{L}(\mathcal{P})&\leq \frac{t p^d}{p^m}+\frac{1}{|(G^*_{p,m})|^t}\sum_{\bsq(X)\in (G^*_{p,m})^t }\sum_{\bsk \in\Delta_m^*}\rho_{\wal}(\bsk)\left|\sum_{l=0}^{p^d-1}\wal_{\bsk}(\bsy_{n_l})\right|.
\end{align*}
We concentrate on 
$$\left|\sum_{l=0}^{p^d-1}\wal_{\bsk}(\bsy_{n_l})\right|=\left|\sum_{l=0}^{p^d-1}e\left( \sum_{i=1}^t\sum_{j=1}^my_{l,j}^{(i)}k_{j-1}^{(i)}\right)\right|.$$
where we expanded $k_i=\sum_{j=0}^{m-1}k^{(i)}_jp^j$ in base $p$. We abbreviate the $j$th row of $C_{i,d}$ to $\bsc_j^{(i)}$ and remember that 
$$y_{l,j}^{(i)}=r_j^{(i)}+\bsc_j^{(i)}\cdot \bsl \pmod{p}.$$
Hence 
$$\left|\sum_{l=0}^{p^d-1}\wal_{\bsk}(\bsy_{n_l})\right|=\left|\sum_{l=0}^{p^d-1}e\left(\left(\sum_{i=1}^t\sum_{j=1}^m k_{j-1}^{(i)}\bsc_j^{(i)}\right)\cdot \bsl\right)\right|=\left\{ 
\begin{array}{ll}
p^d & \mbox{ if }C_{1,d}^{T}\rmk_1+\cdots + C_{t,d}^{T}\rmk_t=0\in\FF_p^d\\
\bszero&\mbox{ else.}
\end{array}
\right.$$
Here $\rmk_i$ denotes the $m$-dimensional column vector $(k_0^{(i)},\ldots,k_{m-1}^{(i)})^T$ built up by the base $p$ digits of the $i$th component of $\bsk$. 

Now a crucial point is that 
\begin{equation}\label{equ:equivalence1}C_{1,d}^{T}\rmk_1+\cdots + C_{t,d}^{T}\rmk_t=\bszero\in\FF_p^d\end{equation}
 is equivalent to \begin{equation}\label{equ:equivalence2}\nu\left(\left\{\frac{\bsk(X)B(X)\bsq(X)}{p(X)}\right\}\right)<-d.\end{equation}
 
Note that \eqref{equ:equivalence1} denotes 
 
 $$\sum_{i=1}^t \begin{pmatrix}
a^{(i)}_1&\cdots&a^{(i)}_m\\
\vdots&\ddots&\vdots\\
a^{(i)}_d&\cdots &a^{(i)}_{m+d-1}
\end{pmatrix} 
\cdot\begin{pmatrix}k_0^{(i)}\\ \vdots\\ k_{m-1}^{(i)} \end{pmatrix} =\bszero \in \FF_p^d.
$$
Following the argumentations of \cite[Proof of Lemma~10.6]{DP10}
we end up with 
$$\frac{\bsk(X)B(X)\bsq(X)}{p(X)}=g+H$$
with $g\in\FF_p[X]$ and $H\in\FF_p((X^{-1}))$ of the form $\sum_{j=d+1}^\infty h_jX^{-j}$ which is equivalent to \eqref{equ:equivalence2}.

We define 
$$\mathcal{D}'_{\bsq,p,B}=\{\bsk\in\{0,1,\ldots,p^{m}-1\}^t:\nu\left(\left\{\frac{\bsk(X)\cdot B(X)\cdot \bsq(X)}{p(X)}\right\}\right)< -d\}\setminus\{\bszero\}.$$
and its subset
$$\mathcal{D}'_{\bsq,p}=\{\bsk\in\{0,1,\ldots,p^{m}-1\}^t:\bsk(X)\cdot \bsq(X)\equiv 0 \pmod{p(X)}\}\setminus\{\bszero\}.$$
Using the above considerations we obtain for 
\begin{align*}
\frac{1}{|(G^*_{p,m})|^t}\sum_{\bsq(X)\in (G^*_{p,m})^t }\sum_{\bsk \in\Delta_m^*}\rho_{\wal}(\bsk)\left|\sum_{l=0}^{p^d-1}\wal_{\bsk}(\bsy_{n_l})\right|=\sum_{\bsk \in\Delta_m^*}\rho_{\wal}(\bsk)\frac{1}{|(G^*_{p,m})|^t}\sum_{\bsq(X)\in (G^*_{p,m})^t \atop \bsk\in\mathcal{D}'_{\bsq,p,B}}p^d. 
\end{align*}
Altogether we have to compute for $\bsk\in\Delta_m^{*}$ the number
\begin{align*}
\#\{\bsq(X)\in (G^*_{p,m})^t:\bsk\in \mathcal{D}'_{\bsq,p,B}\}&=\underbrace{\#\{\bsq(X)\in (G^*_{p,m})^t:\bsk\in \mathcal{D}'_{\bsq,p}\}}_{=:K_1}\\
&+\underbrace{\#\{\bsq(X)\in (G^*_{p,m})^t:\bsk\in \mathcal{D}'_{\bsq,p,B}\setminus \mathcal{D}'_{\bsq,p}\}}_{=:K_2}.
\end{align*}

The easy part is to compute $K_1$ which equals $(p^m-1)^{t-1}$ (confer, e.g., \cite[Proof of Theorem~10.21]{DP10}). 

We now concentrate on $K_2$. 

Let $t_0$ be maximal such that $k_{t_0}\neq 0$. We denote by $\bsb^{(i)}$ the projection of $\bsb$ onto the first $i$ components of $\bsb$. Then
\begin{eqnarray*}
\lefteqn{\frac{K_2}{(p^m-1)^{t-t_0}}}\\
&=&\#\{\bsq^{(t_0)}(X)\in(G_{p,m}^*)^{t_0}:\nu\left(\left\{\frac{\bsk^{(t_0)}(X) B(X)\bsq^{(t_0)}(X)}{p(X)}\right\}\right)<-d\\
&& \mbox{ and }\bsk^{(t_0)}(X)\cdot\bsq^{(t_0)}(X)\not\equiv 0\pmod{p(X)}\}\\
&\leq &\sum_{\bsq^{(t_0-1)}(X)\in(G_{p,m}^*)^{t_0-1}}\#\{q_{t_0}(X)\in G_{p,m}:\\
&&\nu(\{\big(\bsk^{(t_0-1)}(X) B(X)\bsq^{(t_0-1)}(X)+k_{t_0}(X)B(X)q_{t_0}(X)\big)/p(X)\})<-d\\
&& \quad \mbox{ and }\bsk^{(t_0-1)}(X)\cdot\bsq^{(t_0-1)}(X)\not\equiv -k_{t_0}(X)q_{t_0}(X)\pmod{p(X)}\}.
\end{eqnarray*}
Since $p(X)$ is irreducible there is exactly one $a(X)\in G_{p,m}$ such that $$\bsk^{(t_0-1)}(X)\cdot\bsq^{(t_0-1)}(X)\equiv -k_{t_0}(X)a(X)\pmod{p(X)}.$$
Thus 
\begin{eqnarray*}
\lefteqn{\frac{K_2}{(p^m-1)^{t-t_0}}}\\
&\leq & \sum_{\bsq^{(t_0-1)}(X)\in(G_{p,m}^*)^{t_0-1}}\#\{q_{t_0}(X)\in G_{p,m}\setminus\{a(X)\}:\\
&&\nu(\{(\bsk^{(t_0-1)}(X) B(X)\bsq^{(t_0-1)}(X)+k_{t_0}(X)B(X)q_{t_0}(X))/p(X)\})<-d\}.
\end{eqnarray*}
Now as $q_{t_0}(X)$ runs through $G_{p,m}\setminus\{a(X)\}$, $$\bsk^{(t_0-1)}(X)\cdot\bsq^{(t_0-1)}(X)+ k_{t_0}(X)q_{t_0}(X)\pmod{p(X)}$$
runs through all polynomials in $G_{p,m}^*$. 
As $B(X)$ and $p(X)$ were assumed coprime we have that 
$$B(X)\bsk^{(t_0-1)}(X)\cdot\bsq^{(t_0-1)}(X)+ B(X)k_{t_0}(X)q_{t_0}(X)\pmod{p(X)}$$
runs through all polynomials in $G_{p,m}^*$. 

Hence 
\begin{eqnarray*}
\frac{K_2}{(p^m-1)^{t-t_0}}
&\leq & \sum_{\bsq^{(t_0-1)}\in(G_{p,m}^*)^{t_0-1}}\#\{b(X)\in G_{p,m}^*:\nu(b(X)/p(X))<-d\}.
\end{eqnarray*}
Altogether the core estimate provides Lemma~\ref{lem:4} which states
$$\#\{b(X)\in G_{p,m}^*:\nu(b(X)/p(X))<-u\}\leq p^{m-d}-1. $$

%applying Lemma~\ref{lem:3} we obtain
Thus
$$K_2\leq  (p^m-1)^{t-1}(p^{m-d}-1). $$
So we can summarize 
$$K_1+K_2\leq (p^m-1)^{t-1}p^{m-d}.$$
Finally, application of Lemma~\ref{lem:3} yields
$$\frac{1}{|G_{p,m}^*|^t}\sum_{\bsq\in(G_{p,m}^*)^t}LD_L^*(\mathcal{P})\leq t+\frac{p^m}{p^m-1}\left(1+m\frac{p^2-1}{3p}\right)^t=O_{p,t}(\log^t N).$$

\section{Proof of Theorem~\ref{thm:2}}\label{sec:4}

The proof follows the same steps as the proof of Theorem~\ref{thm:1}, until we have to compute the average
$$\frac{1}{|(G^*_{p,m})|}\sum_{g(X)\in (G^*_{p,m})} L_kD^*_{L_k}(\mathcal{P}_k).$$
We show again that it is of the form $O_{p,t}(\log^{t} N)$. 

Using the same argumentation as in the proof of Theorem~\ref{thm:1} we end up with treating 

\begin{align*}
\#\{g(X)\in G^*_{p,m}:\bsk\in \mathcal{D}'_{g,p,B,t}\}&=\underbrace{\#\{g(X)\in G^*_{p,m}:\bsk\in \mathcal{D}'_{g,p,t}\}}_{=:K_1}\\
&+\underbrace{\#\{g(X)\in G^*_{p,m}:\bsk\in \mathcal{D}'_{g,p,B,t}\setminus \mathcal{D}'_{g,p,t}\}}_{=:K_2}
\end{align*}
where 

$$\mathcal{D}'_{g,p,B,t}=\{\bsk\in\{0,1,\ldots,p^{m}-1\}^t:\nu\left(\left\{\frac{\bsk(X)\cdot B(X)\cdot (g(X),g^2(X)\ldots,g^{t}(X))}{p(X)}\right\}\right)< -d\}\setminus\{\bszero\}.$$
and 
$$\mathcal{D}'_{q,p,t}:=\{\bsk\in\{0,1,\ldots,p^{u}-1\}^t:\bsk(X)\cdot (g(X),g^2(X)\ldots,g^{t}(X))\equiv 0 \pmod{p(X)}\}\setminus\{\bszero\}.$$

The easy part is again to estimate $K_1$, which is $\leq t$, since 
$$k_1(X)Y+k_2(X)Y^2+\cdots+k_t(X)Y^{t}\equiv 0 \pmod{p(X)}$$ 
has at most $t$ solutions for $Y$ modulo $p(X)$. 

In the following we show that $K_2\leq t(p^{m-d}-1)$.

We know that for each $a(X)\in G_{p,m}^*$ the congruence 
$$k_1(X)Y+k_2(X)Y^2+\cdots+k_{t}(X)Y^{t}\equiv a(X) \pmod{p(X)}$$ 
has at most $t$ solutions for $Y$ modulo $p(X)$. 
As $B(X)$ and $p(X)$ are coprime 

$$B(X)k_1(X)Y+B(X)k_2(X)Y^2+\cdots+B(X)k_{t}(X)Y^{t}\equiv b(X) \pmod{p(X)}$$ 

has at most $t$ solutions for each $b(X)\in G_{p,m}^*$. By Lemma~\ref{lem:4} only $p^{m-d}-1$ values of $b(X)$ have to be considered. 
Hence we have $K_2\leq t(p^{m-d}-1)$. 

Then the result follows exactly by the same arguments as in the proof of Theorem~\ref{thm:1}. 

\begin{rem}{\rm
Note that bounding $K_2$ in the proof of Theorem~\ref{thm:2} would not work if we consider generating tuples of the form $(1,g(X),\ldots,g^{t-1}(X))$ instead of $(g(X),g^2(X),\ldots,g^{t}(X))$. This is the reason why we defined Korobov polynomial lattice point sets in this way. 
Furthermore, mixing a polynomial point set $\mathcal{P}(1,p(X))$ with the first component $(n/p^m)_{n=0,1,\ldots,p^m-1}$ won't result in good discrepancy bounds. 
}
\end{rem}

\section*{Acknowledgments}

The author is supported by the Austrian Science Fund (FWF):
Project F5505-N26, which is a part of the Special Research Program
``Quasi-Monte Carlo Methods: Theory and Applications''.

\end{document}